\documentclass[12pt]{amsart}

\usepackage[a4paper,left=2.5cm,right=2.5cm,top=2.5cm,bottom=2.5cm]{geometry}

\usepackage{amsmath}

\usepackage{graphicx}
\usepackage{amssymb}
\usepackage{enumerate}
\usepackage{setspace}
\usepackage{color}

\newtheorem{theorem}{Theorem}[section]
\newtheorem{lemma}[theorem]{Lemma}
\newtheorem{proposition}[theorem]{Proposition}
\newtheorem{corollary}[theorem]{Corollary}
\newtheorem{remark}[theorem]{Remark}
\numberwithin{equation}{section}

\begin{document}
\title{Projections of Gibbs measures on self-conformal sets}

\author{Catherine Bruce}
\address{School of Mathematics, University of Manchester, Oxford Road, Manchester M13 9PL, United Kingdom}
\email{catherine.bruce@manchester.ac.uk}

\author{Xiong Jin}
\address{School of Mathematics, University of Manchester, Oxford Road, Manchester M13 9PL, United Kingdom}
\email{xiong.jin@manchester.ac.uk}

\begin{abstract}
We show that, under an ergodicity assumption on the rotation skew product, every orthogonal projection of a Gibbs measure on a self-conformal set in $\mathbb{R}^d$ ($d\ge2$) has the maximum possible Hausdorff dimension. No separation condition is required. As a corollary we prove the dimension form of Falconer's distance set conjecture for this class of self-conformal sets supporting a Gibbs measure of full Hausdorff dimension, in particular those satisfying the open set condition.
\end{abstract}

\maketitle

\vspace{20pt}
\section{Introduction}
Let $d\ge 2$ and $1\le k\le d$ be integers, let $K\subset\mathbb{R}^d$ be Borel or analytic, and let $\Pi_{d,k}$ be the set of orthogonal projections from $\mathbb{R}^d$ to its $k$-dimensional subspaces, with its natural rotation-invariant probability measure $\xi$. Then
\[
\dim_H\pi K=\min\{k,\dim_HK\}\text{ for }\xi\text{-almost every }\pi\in\Pi_{d,k}.
\]
This result, due to Marstrand \cite{Marstrand} in the plane and Mattila \cite{Mattila} in $\mathbb{R}^d$, has been the basis for a great deal of work in the field of fractal geometry. Much of this work concerns general Borel sets $K$ and almost all projections $\pi\in \Pi_{d,k}$. Furstenberg's CP-chain method \cite{Furs} enabled Hochman and Shmerkin \cite{HochShmer} to show that for self-similar sets and measures with {\em dense rotations} satisfying the strong separation condition, the result holds for all $\pi\in\Pi_{d,k}$. For subsequent developments, see the survey papers \cite{FFJ,Shmerkin,MattSur} and the references therein. In particular, Falconer and Jin \cite{FalcJin} extended their result to random cascade measures (including self-similar measures as special cases) without requiring any separation condition. This was achieved by working on symbolic space, where overlaps do not obstruct the cylinder partitions used in the construction of the zoom-in process.

In \cite{HochShmer} Hochman and Shmerkin also considered the projections of products of Gibbs measures on one-dimensional non-linear Cantor sets. The authors used the so-called {\em limit diffeomorphisms} of one-dimensional non-linear iterated function systems developed by Sullivan \cite{Sullivan} and Bedford and Fisher \cite{BF} to transfer the problem back to the affine case. For planar conformal systems satisfying the strong separation condition, Fraser and Pollicott \cite[Theorem 2.3]{FraserPoll} proved that a conformal image of a measure equivalent to a cylinder restriction of the Gibbs measure generates an ergodic CP-chain. This does not directly yield the strong projection theorem for the original measure: conformal maps do not preserve the dimensions of orthogonal projections, and the required rotation condition must also be verified.

The main difficulty in studying scenery flows of self-conformal measures is that linear rescalings do not agree with the non-linear maps of the IFS. We follow the symbolic local-entropy approach of \cite{FalcJin,HochShmer}. Passing to the natural extension and conditioning on the infinite past, as in \cite{BruceJin}, gives conditional Gibbs measures with an exact update rule. Bounded distortion provides conformal coordinates indexed by the past in which this rule is realised by similarities. A suspension flow then supplies the fixed logarithmic scales needed for the local entropy averages argument. The conditional Gibbs measures are uniformly equivalent to the original measure, allowing us to transfer the resulting dimension statements back to it.

We consider an iterated function system (IFS) of conformal $C^{1+\epsilon}$ maps $\mathcal{I}=\{f_i\}_{i=1}^m$ ($m\ge2$), for some $\epsilon>0$, in $\mathbb{R}^d$, satisfying the following assumptions:
\begin{itemize}
\item[(A0)] There is a bounded, convex open set $U\subset \mathbb{R}^d$ such that each $f_i:U\to U$ is injective and its derivative at every $x\in U$ is a positive scalar multiple of a rotation matrix:
\begin{equation}\label{eqn:ifs}
f_i'(x)=r_i(x)O_i(x),
\end{equation}
where $r_i(x)\in (0,\infty)$ and $O_i(x)\in SO(d,\mathbb{R})$.
\item[(A1)] There exists a constant $0<r^*<1$ such that $r_i(x)\le r^*$ for every $1\le i \le m$ and $x\in U$.
\end{itemize}

\begin{remark}
Convexity allows us to use line segments in the distortion estimates. Extensions to non-convex domains require corresponding uniform estimates along curves; we retain the convexity assumption here.
\end{remark}

These assumptions imply that the IFS $\mathcal{I}$ is uniformly contractive on $U$. We assume throughout that it has a non-empty compact attractor $K\subset U$, necessarily unique, satisfying
\begin{equation}
\label{eqn:attractor}
K=\bigcup_{i=1}^mf_i(K).
\end{equation}
Such a $K$ is called a \emph{self-conformal set}. The set $K$ has a natural symbolic representation: let $\Lambda=\{1,\ldots,m\}$ be the alphabet and let $\Lambda^\mathbb{N}$ be the symbolic space with $m$ letters. For $i=i_1\cdots i_n\in\Lambda^n$, write
\[
f_i=f_{i_1}\circ \cdots \circ f_{i_n}
\]
and, for $\underline{i}=i_1i_2\cdots\in\Lambda^\mathbb{N}$ and $n\ge1$, write $\underline{i}|_n=i_1\cdots i_n$. Fix a point $x_0\in U$. Then we may define a map $\Phi:\Lambda^\mathbb{N}\to U$ by
\[
\Phi(\underline{i})=\lim_{n\to \infty} f_{\underline{i}|_n}(x_0).
\]
By uniform contractivity, the above limit exists and it does not depend on the choice of $x_0$. Its image is the self-conformal set $K$ and $\Phi$ is called the canonical map. We also write $x_{\underline{i}}=\Phi(\underline{i})$ for the point corresponding to $\underline{i}\in\Lambda^{\mathbb{N}}$.

For $i=i_1\cdots i_n\in \Lambda^n$ denote by $[i]=\{\underline{i}\in\Lambda^\mathbb{N}:\underline{i}|_n=i\}$ the cylinder in $\Lambda^\mathbb{N}$ encoded by $i$. Let $\mathcal{B}$ denote the $\sigma$-algebra generated by cylinders. Let $\sigma$ denote the left-shift operator on $\Lambda^\mathbb{N}$. Let $\varphi:\Lambda^\mathbb{N} \to \mathbb{R}$ be a H\"older potential on $\Lambda^\mathbb{N}$ and let $\mu_\varphi$ denote its Gibbs measure (see Subsection \ref{subsec:gibbs} for a precise definition). We are interested in the orthogonal projections of the push-forward measure
\[
\Phi\mu_\varphi=\mu_\varphi\circ \Phi^{-1}
\]
on the self-conformal set $K$. Before stating our main result, we shall present an analogue of the {\em dense rotations} condition in the self-conformal case. We discuss this condition in greater detail in Subsection \ref{sec:cge}. Let $G=SO(d,\mathbb{R})$ and define a map $\phi:\Lambda^\mathbb{N}\to G$ as follows:
\[
\phi(\underline{i})=O_{i_1}(\Phi(\sigma \underline{i})) \text{ for } \underline{i}=i_1i_2\cdots.
\]
The rotation $\phi(\underline{i})$ records the orthogonal part of $Df_{i_1}$ at $\Phi(\sigma\underline{i})$. We may define the skew product $\sigma_\phi:\Lambda^\mathbb{N}\times G\to \Lambda^\mathbb{N}\times G$ as
\[
\sigma_\phi(\underline{i},O)=(\sigma \underline{i}, O\phi(\underline{i})).
\]
We assume
\begin{itemize}
\item[(A2)] $\sigma_\phi$ has a dense orbit in $\Lambda^\mathbb{N}\times G$.
\end{itemize}
By a compact group extension theorem (see Subsection \ref{sec:cge}), this implies that the dynamical system
\begin{equation}
\label{eqn:cgeds}
(\Lambda^\mathbb{N}\times G,\mathcal{B}\otimes\mathcal{B}_G,\mu_\varphi\times\xi,\sigma_\phi)
\end{equation}
is ergodic, where $\xi$ is the normalised right-invariant Haar measure on $G$, and $\mathcal{B}_G$ is its Borel $\sigma$-algebra.
\begin{remark}
\label{rmk:toptrans}
The assumption (A2) is used only to prove that the dynamical system \eqref{eqn:cgeds} is ergodic. In this dynamical system, (A2) is equivalent to topological transitivity, that is, for any non-empty open sets $U,V\subset\Lambda^{\mathbb{N}}\times G$, there exists $n\ge1$ such that $\sigma_{\phi}^n(U)\cap V$ is non-empty. See \cite{Silverman} for details. In the planar case, an irrational rotation at a periodic point gives a sufficient condition; see Lemma \ref{lemma:doimpltt}.
\end{remark}
We can now state our main result:

\begin{theorem}\label{thm:main}
Under assumptions (A0), (A1) and (A2), for all $\pi\in \Pi_{d,k}$ we have
\[
\dim_H \pi \Phi\mu_\varphi=\min\{k,\dim_H \Phi\mu_\varphi\}.
\]
\end{theorem}

With the same approach as in \cite{HochShmer}, using local linear approximations to $C^1$-maps, we can also prove the following. Here a $C^1$-map on $K$ means the restriction of a $C^1$-map defined on a neighbourhood of $K$, and having no singular points means that its derivative has rank $k$ at every point of $K$.

\begin{corollary}\label{thm:nonli}
Under assumptions (A0), (A1) and (A2), for all $C^1$-maps $h:K\to \mathbb{R}^k$ without singular points,
\[
\dim_H h \Phi\mu_\varphi=\min\{k,\dim_H \Phi\mu_\varphi\}.
\]
\end{corollary}

In order to apply Theorem \ref{thm:main} to self-conformal sets, we consider when such sets support a Gibbs measure of full dimension. We say that the IFS satisfies the \emph{strong variational principle for Gibbs measures} if there exists a H\"older potential $\varphi$ on $\Lambda^{\mathbb{N}}$ whose Gibbs measure $\mu_\varphi$ satisfies
\[
\dim_H\Phi\mu_\varphi=\dim_H K.
\]
Under this strong variational principle, we obtain the following two corollaries.

\begin{corollary}\label{coro:nonliproj}
Under assumptions (A0), (A1) and (A2), and assuming that $K$ supports a Gibbs measure of full dimension, for all $\pi\in\Pi_{d,k}$,
\begin{equation}
\label{eqn:mainsetsproj}
\dim_H\pi(K)=\min\{k,\dim_H K\}.
\end{equation}
\end{corollary}

\begin{corollary}\label{coro:nonli}
Under assumptions (A0), (A1) and (A2), and assuming that $K$ supports a Gibbs measure of full dimension, for all $C^1$-maps $h:K\to \mathbb{R}^k$ without singular points,
\begin{equation}
\label{eqn:mainsets}
\dim_H h(K)=\min\{k,\dim_HK\}.
\end{equation}
\end{corollary}

The lower bounds follow by applying Theorem \ref{thm:main} and Corollary \ref{thm:nonli} to a Gibbs measure of full dimension; the upper bounds follow from the Lipschitz property of the maps on $K$.
If $K$ satisfies the open set condition (OSC), then there exists a Gibbs measure $\mu$ of a H\"older potential on $\Lambda^\mathbb{N}$ such that $\dim_H \Phi\mu=\dim_HK$, and $\Phi\mu$ is equivalent to $\dim_HK$-dimensional Hausdorff measure restricted to $K$ (see \cite{Bowen2,PRSS}).
Feng and Hu \cite[Theorem 2.13]{FengHu} proved a variational principle expressing $\dim_H K$ as the supremum of the dimensions of the push-forwards of ergodic measures on $(\Lambda^{\mathbb{N}},\sigma)$.
Under the asymptotically weak separation condition, this supremum is attained by an ergodic measure \cite[Remark 8.2]{FengHu}.
These statements do not assert that the full-dimensional measure is Gibbs for a H\"older potential, as required by the strong variational principle for Gibbs measures above.

The application of a strong projection theorem to distance sets is due to Orponen \cite{Orponen}. The local linearisation argument also applies to the distance function away from its centre, giving the following dimension form of Falconer's distance set conjecture (see also \cite{Shmerkin2}).

\begin{corollary}\label{coro:FDC}
Under assumptions (A0), (A1) and (A2), and assuming that $K$ supports a Gibbs measure of full dimension, if $\dim_H K\ge1$, then for every $a\in K$,
\[
\dim_H \{|x-a|:x\in K\}=\dim_H \{|x-y|: x,y\in K\} =1.
\]
\end{corollary}
\begin{proof}
Let $\nu=\Phi\mu$ be the full-dimensional Gibbs measure. Since $\nu$ is exact dimensional and $\dim_H\nu\ge1$, it has no atoms. Fix $a\in K$. In the notation of Proposition \ref{prop:nonli}, choose a typical $z$ and apply its local linearisation argument to
\[
g(y)=|A_z^{-1}(y)-a|,\qquad y\in A_z(K).
\]
This function is Lipschitz and is $C^1$ away from the single $M(z)$-null point $A_z(a)$. Its derivative has rank one there, with its nonzero singular value bounded below, since $A_z^{-1}$ is conformal with derivative bounded above and below near $A_z(K)$. For each other point the sufficiently small cylinders lie in its smooth domain, so the entropy comparison in that proof applies, with a uniform comparison constant. It follows that $\dim_H gM(z)=1$. Equivalence of $\mu_p$ and $\mu$ gives $\dim_H(|\cdot-a|\nu)=1$, proving the pinned assertion and hence the unpinned assertion.
\end{proof}

\begin{remark}
The unpinned conclusion in Corollary \ref{coro:FDC}, for $d=2$ and $\dim_HK>1$, was proved by Fraser and Pollicott \cite[Theorem 2.7]{FraserPoll}. Their result does not require assumption (A2).

\end{remark}

We give a planar example satisfying the geometric and dynamical assumptions (A0)--(A2). According to Theorem 14.15 in \cite{Falc}, when $|c|>\frac{1}{4}(5+2\sqrt{6})\approx2.475$, the Julia set $J_{f_c}$ defined by the quadratic polynomial $f_c(z)=z^2+c$ is totally disconnected, and is the attractor of the conformal IFS
\begin{equation}\label{eqn:juliaconf}
\{f_1(z)=\sqrt{z-c},\;\;f_2(z)=-\sqrt{z-c}\}.
\end{equation}
Take $U=\{z:|z|<|2c|^{1/2}\}$ and choose a holomorphic branch of $\sqrt{z-c}$ on $U$. Such a branch exists because $c\notin\overline U$. The set $U$ is bounded, open and convex. Since $|c|>2$, we have $|2c|^{1/2}<|c|$, and therefore $|f_i(z)|^2\le |z|+|c|<2|c|$ for $z\in U$. Thus $f_i(U)\subset U$. Moreover,
\[
|f_i'(z)|=\frac{1}{2}|z-c|^{-1/2}\le \frac{1}{2}(|c|-|2c|^{1/2})^{-1/2}<1.
\]
The two branch images are disjoint: $f_1(z)=f_2(w)$ would imply $z=w$ after squaring, and then $f_1(z)=0$, which is impossible on $U$. This verifies (A0), (A1) and the OSC. A consequence of Lemma \ref{lemma:doimpltt} is that to achieve (A2) it is sufficient to show that there exists a finite word $u$ such that $O_u(x_{\overline{u}})$ is an irrational rotation, where $\overline{u}=uuu\cdots$ denotes the periodic infinite word of $u$. For this example, we give a sufficient condition on $c$ for such a word to exist. Take the fixed point $\alpha=\frac{1+\sqrt{1-4c}}{2}\in J_{f_c}$ and label the inverse branches so that $f_1(\alpha)=\alpha$. We have
\[
f_1'(\alpha)=\frac{1}{2\sqrt{\alpha-c}}=\frac{1}{2\alpha}=\frac{1}{1+\sqrt{1-4c}}.
\]
Therefore, $\arg f_1'(\alpha)/\pi$ is irrational if and only if
\begin{equation}
\label{eqn:argirr}
\frac{\arg(1+\sqrt{1-4c})}{\pi}\notin\mathbb{Q}.
\end{equation}
Under this condition, $\{O_{\underline{i}|_n}(x_{\sigma^n\underline{i}})\}_{n\ge1}$ is dense in $SO(2,\mathbb{R})$ for $\underline{i}=111\cdots$. Lemma \ref{lemma:doimpltt} therefore gives (A2).

\medskip

\begin{remark}
The example above belongs to the class of hyperbolic Julia sets. Bedford, Fisher and Urba\'nski \cite[Theorem 4.9]{BFU} proved that the full scenery flow of a hyperbolic rational map is ergodic unless one of the following holds:
\begin{enumerate}[i)]
\item the Julia set $J_f$ is a geometric circle and $f$ is biholomorphically conjugate to a finite Blaschke product, or
\item the Julia set $J_f$ is totally disconnected and $J_f$ is contained in a real-analytic curve with self-intersections (if any) lying outside the Julia set.
\end{enumerate}
In the argument below, the passage from the compact group extension to fixed scales is made through an ergodic suspension. We do not require invariance of the original Gibbs measure under a variable-length stopping-word shift.
\end{remark}

Section \ref{sec:prelim} gives the symbolic, distortion and Gibbs-measure preliminaries, including the conditional measures on the natural extension. In Section \ref{sec:dimofprojs} we construct the zoom-in flow, establish the entropy lower bound at fixed scales, and deduce the projection and $C^1$-image theorems.

\section{Preliminaries}
\label{sec:prelim}
\subsection{Symbolic space}
Let $\Lambda=\{1,\ldots,m\}$ be an alphabet with $m\ge2$ symbols. Let $\Lambda^*=\bigcup_{n\ge0}\Lambda^n$ be the set of finite words. For $i\in \Lambda^*$ let $|i|$ denote the length of the word. Let $\Lambda^{\mathbb{N}}$ be the symbolic space of infinite sequences from the alphabet. For $\underline{i}\in\Lambda^{\mathbb{N}}$ and $n\ge1$, let $\underline{i}|_n\in\Lambda^n$ be its first $n$ symbols. For $i\in\Lambda^n$, let $[i]=\{\underline{i}\in\Lambda^{\mathbb{N}}:\underline{i}|_{n}=i\}$ be the \emph{cylinder} rooted at $i$. We may endow $\Lambda^{\mathbb{N}}$ with the standard metric $d_{\rho}$ with respect to a real number $\rho\in(0,1)$, that is, for $\underline{i}, \underline{j}\in\Lambda^{\mathbb{N}},$
\[
d_{\rho}(\underline{i}, \underline{j})=\rho^{\sup\{n\ge0:\underline{i}|_n=\underline{j}|_n\}},
\]
with the conventions that $\underline{i}|_0=\emptyset$, $\rho^\infty=0$ and $\Lambda^0=\{\emptyset\}$.
Then $(\Lambda^{\mathbb{N}},d_{\rho})$ is a compact metric space. Let $\mathcal{B}$ be its Borel $\sigma$-algebra. Define the left shift map $\sigma$ by $\sigma(\underline{i})=(i_{n+1})_{n\ge1}$ for $\underline{i}=(i_n)_{n\ge1}\in\Lambda^{\mathbb{N}}.$

\subsection{Self-conformal sets}
Let $\mathcal{I}$ be the IFS in \eqref{eqn:ifs}, with non-empty compact attractor $K\subset U$. Since $U$ is convex, we may choose a convex open set $V$ with $K\subset V\subset\overline V\subset U$. All the sets $f_i(\overline V)$, $i\in\Lambda^*$, lie in a common compact subset $W$ of $U$. Indeed, the maximum distance of their points from $K$ tends to zero uniformly as $|i|\to\infty$, while there are only finitely many shorter words. We may enlarge $W$ to its convex hull, which is still compactly contained in $U$. Recall that $\Phi:\Lambda^{\mathbb{N}}\to K$ is the canonical projection, that is, $\Phi(\underline{i})=\lim_{n\to\infty}f_{\underline{i}|_n}(x_0)$ for some $x_0\in U$. We shall also use the notation $x_{\underline{i}}=\Phi( \underline{i})$ for $\underline{i}\in \Lambda^\mathbb{N}$.

\subsection{Bounded distortion}
For $x\in V$ and $i=i_1\cdots i_n \in\Lambda^n$ we may write
\begin{equation}\label{eqn:taylor}
f_i'(x)=r_i(x)O_i(x),
\end{equation}
where
\begin{align*}
r_i(x)&=r_{i_1}(f_{i_2}\circ\cdots\circ f_{i_n}(x))\cdot r_{i_2}(f_{i_3}\circ\cdots\circ f_{i_n}(x))\cdot\cdot\cdot r_{i_n}(x),\\
O_i(x)&=O_{i_1}(f_{i_2}\circ\cdots\circ f_{i_n}(x))\cdot O_{i_2}(f_{i_3}\circ\cdots\circ f_{i_n}(x))\cdot\cdot\cdot O_{i_n}(x).
\end{align*}
Conformal maps in dimension $d\ge2$ are smooth. On the compact set $W$, the functions $\log r_i$, $i\in\Lambda$, have uniformly bounded derivatives. Consequently, we have the following {\em bounded distortion property}: there exists a constant $C_1\ge1$ such that for all $i\in \Lambda^*$ and $x,y\in V$,
\begin{equation}\label{eqn:bp}
\frac{r_i(x)}{r_i(y)} \le C_1.
\end{equation}
To see this, for $i=i_1\cdots i_n\in \Lambda^n$ we have
\begin{equation}\label{eqn:rij}
\log \frac{r_i(x)}{r_i(y)} =\sum_{k=1}^n\left(\log r_{i_k}(f_{i_{k+1}\cdots i_{n}}(x))-\log r_{i_k}(f_{i_{k+1}\cdots i_{n}}(y))\right).
\end{equation}
By the smoothness of $\log r_i$ one can find a constant $\widetilde{C}_1$ such that for all $x,y\in W$ and $i\in\Lambda$,
\[
|\log r_i(x)-\log r_i(y)|\le \widetilde{C}_1 |x-y|.
\]
On the other hand, by (A1), one can find another constant $\widetilde{C}_1'$ such that for all $x,y\in V$, $n\ge0$ and $i=i_1\cdots i_n\in\Lambda^n$,
\[
|f_i(x)-f_i(y)|\le \widetilde{C}_1' (r^*)^n.
\]
Combining these two inequalities we get from \eqref{eqn:rij} that
\[
\log \frac{r_i(x)}{r_i(y)} \le \widetilde{C}_1\widetilde{C}_1' \sum_{k=1}^n (r^*)^{n-k} \le \widetilde{C}_1\widetilde{C}_1' \frac{1}{1-r^*}:=\log C_1.
\]
The {\em bounded distortion} also implies the following fact: there exists a constant $C_2>0$ such that for all $i\in \Lambda^*$ and all $x,y\in V$,
\begin{equation}\label{eqn:diam}
C_2^{-1}\overline{r}_i|x-y|\le |f_i(x)-f_i(y)|\le C_2\overline{r}_{i}|x-y|,
\end{equation}
where for $i\in\Lambda^*$ we denote by $\overline{r}_i=\sup\{r_i(x):x\in V\}$. For a proof see \cite[Lemma 2.2]{Pat} for example (note that the proof does not require any separation condition).

\subsection{Gibbs measures}
\label{subsec:gibbs}
Let $\varphi$ be a \emph{H\"older potential} defined on $\Lambda^{\mathbb{N}}$. This means there exist constants $\kappa>0$ and $\beta\in(0,1)$ such that for $n\ge 1$,
\begin{equation}\label{eqn:beta}
\mathrm{Var}_n(\varphi):=\sup_{i\in \Lambda^n}\sup_{\underline{i},\underline{j}\in [i]} |\varphi(\underline{i})-\varphi(\underline{j})| \le \kappa \beta^n.
\end{equation}
For $n\ge1$ the \emph{$n$th-order Birkhoff sum} of $\varphi$ over $\sigma$ is defined as
\[
S_n\varphi(\underline{i})=\sum_{k=0}^{n-1}\varphi\circ\sigma^k(\underline{i}),
\]
for $\underline{i}\in\Lambda^{\mathbb{N}}$. The \emph{topological pressure} of $\varphi$ on $\Lambda^{\mathbb{N}}$ is given by
\[
P(\varphi)=\lim_{n\to\infty}\frac{1}{n}\log\sum_{j\in\Lambda^n}\exp\left(\max_{\underline{i}\in[j]}S_n\varphi(\underline{i})\right),
\]
where the limit exists by subadditivity of the logarithm of the partition sum. It follows from the thermodynamic formalism developed by Sinai, Ruelle, Bowen and Walters \cite{Bowen,Ruelle} that there exists a unique shift-invariant probability measure $\mu_\varphi$, called the \emph{Gibbs measure} of $\varphi$, satisfying the following estimate; this measure is ergodic. There is a constant $C_G\ge1$ such that, for every $\underline{i}\in\Lambda^{\mathbb{N}}$, $n\ge1$ and $\underline{j}\in[\underline{i}|_n]$,
\begin{equation}
\label{eqn:gibbs}
C_G^{-1}\le\frac{\mu_{\varphi}([\underline{i}|_n])}{\exp(S_n\varphi(\underline{j})-nP(\varphi))}\le C_G.
\end{equation}
Also, $\mu_{\varphi}$ possesses the \emph{quasi-Bernoulli property}:
\begin{equation}
\label{eqn:quasibern}
C_Q^{-1}\mu_{\varphi}([i])\mu_{\varphi}([j])\le\mu_{\varphi}([ij])\le C_Q\mu_{\varphi}([i])\mu_{\varphi}([j]),
\end{equation}
for all $i,j\in\Lambda^*$, where one may take $C_Q=C_G^3$. This follows by applying \eqref{eqn:gibbs} to $ij$, $i$ and $j$ along the same infinite sequence and its shift. In general, $C_Q$ cannot be replaced by constants tending to one as $|i|,|j|\to\infty$: the dependence across the boundary between two adjacent words persists. For example, for a stationary Markov measure with positive transition matrix $P$ and stationary vector $p$, the ratio $\mu([ij])/(\mu([i])\mu([j]))$ equals $P_{i_{|i|},j_1}/p_{j_1}$, independently of the lengths of the two non-empty words.

For the locally constant potential $\varphi(\underline{i})=\log p_{i_1}$, where $\mathbf p=(p_i)_{i\in\Lambda}$ is a positive probability vector, $P(\varphi)=0$ and
\[
\mu_{\mathbf p}([\underline{i}|_n])
=\exp S_n\varphi(\underline{i})
=p_{i_1}\cdots p_{i_n}.
\]
Thus the Gibbs measure is the Bernoulli measure with weights $\mathbf p$.

\subsection{The natural extension and conditional Gibbs measures}
\label{subsec:natural-extension}
Following the two-sided construction in \cite[Sections 3.2--3.3]{BruceJin}, let $\widehat\mu$ be the natural extension of $\mu=\mu_\varphi$ to $\widehat\Sigma=\Lambda^{\mathbb Z}$. Write $\omega=(p,x)$, where $p=(\ldots,p_{-1},p_0)\in\Sigma^-:=\Lambda^{\mathbb Z_{\le0}}$ is the past and $x=(x_1,x_2,\ldots)\in\Sigma^+:=\Lambda^{\mathbb N}$ is the future. We also write $\omega^-=p$ and $\omega^+=x$, and use $\sigma$ for the two-sided shift. Denote by $\mu^-$ the past marginal. Disintegration over the past gives probability measures $\mu_p$ on $\Sigma^+$ such that
\begin{equation}\label{eqn:past-disintegration}
\widehat\mu(dp,dx)=\mu^-(dp)\mu_p(dx),\qquad
\mu=\int_{\Sigma^-}\mu_p\,\mu^-(dp).
\end{equation}
For a finite word $u$, let $pu$ be the past obtained by appending $u$ and moving the origin to its last symbol. Conditional probabilities are chosen on a common full-measure set on which the following identities hold for all finite words.

\begin{lemma}\label{lem:past-gibbs}
For $\mu^-$-almost every $p$,
\begin{equation}\label{eqn:past-equivalence}
C_Q^{-1}\mu\le\mu_p\le C_Q\mu,
\end{equation}
and, for every finite word $u$,
\begin{equation}\label{eqn:past-covariance}
(\sigma^{|u|})_*\bigl((\mu_p)_{[u]}\bigr)=\mu_{pu}.
\end{equation}
Here $\nu_A=\nu(A)^{-1}\nu|_A$ whenever $\nu(A)>0$, and $H_r$ is the scaling entropy defined in Subsection \ref{de}. Moreover, if $g:\Sigma^+\to\mathbb R^k$ is a Borel map with bounded image, then for every $r>0$,
\begin{equation}\label{eqn:past-entropy}
\left|\int_{\Sigma^-}H_r(g\mu_p)\,\mu^-(dp)-H_r(g\mu)\right|\le\log C_Q.
\end{equation}
\end{lemma}
\begin{proof}
Given a finite past word $v$, the conditional probability of the future cylinder $[u]$ is
\[
\frac{\mu([vu])}{\mu([v])}.
\]
By \eqref{eqn:quasibern} this lies between $C_Q^{-1}\mu([u])$ and $C_Q\mu([u])$. Martingale convergence, first for the countable collection of future cylinders and then by extension to Borel sets, proves \eqref{eqn:past-equivalence}. Shift invariance of $\widehat\mu$ and uniqueness of disintegration give
\[
\mu_p([uv])=\mu_p([u])\mu_{pu}([v]),
\]
which proves \eqref{eqn:past-covariance}. Null sets may be removed simultaneously over the countably many words; the positivity in \eqref{eqn:past-equivalence} ensures that the required appended pasts are also typical.

For \eqref{eqn:past-entropy}, put $\lambda_p=g\mu_p$ and $\lambda=g\mu$. Then $\lambda=\int\lambda_p\,d\mu^-$ and $C_Q^{-1}\lambda\le\lambda_p\le C_Q\lambda$. Consequently,
\begin{align*}
\int H_r(\lambda_p)\,\mu^-(dp)-H_r(\lambda)
=\int\!\int\log\frac{\lambda(B(y,r))}{\lambda_p(B(y,r))}
\,\lambda_p(dy)\,\mu^-(dp).
\end{align*}
The integrand has absolute value at most $\log C_Q$. All these entropies are finite, since the measures have bounded support. This proves the estimate.
\end{proof}
The conditional measures thus evolve by an exact identity. For every fixed map $g$, their averaged image entropy differs from that of the original image measure by a bounded additive term. No quasi-Bernoulli constant tending to one is required.

\subsection{The compact group extension}
\label{sec:cge}
We now consider the compact group extension of $(\Lambda^{\mathbb N},\mathcal B,\mu,\sigma)$, where $\mu=\mu_\varphi$. The rotation coordinate records the changes of direction along successive cylinders. Recall that $G=SO(d,\mathbb R)$ and
\[
\phi(\underline{i})=O_{i_1}(\Phi(\sigma \underline{i})) \text{ for } \underline{i}=i_1i_2\cdots.
\]
The smoothness of the conformal maps and exponential contraction of cylinder diameters imply that $\phi$ is H\"older on $(\Lambda^\mathbb{N},d_\rho)$. We may define the skew product $\sigma_\phi:\Lambda^\mathbb{N}\times G\to \Lambda^\mathbb{N}\times G$ as follows:
\[
\sigma_\phi(\underline{i},O)=(\sigma \underline{i}, O\phi(\underline{i})).
\]
Shift invariance of $\mu$ and right invariance of $\xi$ imply that $\mu\times\xi$ is $\sigma_\phi$-invariant, where $\xi$ is the normalised Haar measure on $G$, that is, the unique right-translation invariant probability measure on $G$. 

\begin{proposition}
The dynamical system $(\Lambda^\mathbb{N}\times G,\mathcal{B}\otimes\mathcal{B}_G,\mu\times\xi,\sigma_{\phi})$ is ergodic.
\end{proposition}
\begin{proof}
The base is a full shift, $\mu$ is the equilibrium measure of a H\"older potential, and $\phi$ is H\"older with values in the compact Lie group $G$. Under (A2), the skew product has a dense orbit, so \cite[Corollary 4.5]{Parry97} gives ergodicity. The right-multiplication convention used here is conjugate to the left-multiplication convention via $O\mapsto O^{-1}$.
\end{proof}
For the planar example in the introduction we use the following sufficient condition for (A2).
\begin{lemma}\label{lemma:doimpltt}
Suppose $d=2$. If a finite word $u$ has fixed point $\alpha=x_{\overline u}$ and $O_u(\alpha)$ is an irrational rotation, then $\sigma_\phi$ is topologically transitive.
\end{lemma}
\begin{proof}
Put $g=f_u$ and $S=O_u(\alpha)$. Since $SO(2)$ is abelian, the chain rule gives
\[
S^{-n}O_{u^n}(x)=\prod_{j=0}^{n-1}S^{-1}O_u(g^j(x)).
\]
The distance of the $j$th factor from the identity decays geometrically, uniformly for $x\in K$, by contraction and the H\"older continuity of $O_u$. Hence these products converge uniformly to a rotation $L(x)$.

Let $[a]\times A$ and $[b]\times B$ be non-empty open rectangles in $\Lambda^{\mathbb N}\times SO(2)$. Fix $\underline t\in[b]$, put $x=\Phi(\underline t)$ and choose $O\in A$. After $|a|+n|u|$ iterates, the point $(au^n\underline t,O)$ has symbolic coordinate $\underline t$ and rotation coordinate
\[
O\,O_a(g^n(x))\,S^n\bigl(S^{-n}O_{u^n}(x)\bigr).
\]
Since $g^n(x)\to\alpha$, the distance of this rotation from $O\,O_a(\alpha)S^nL(x)$ tends to zero. The latter sequence is dense in $SO(2)$, and visits every non-empty open set infinitely often, because $S$ is irrational. Thus an iterate belongs to $[b]\times B$, proving topological transitivity.
\end{proof}

\subsection{Dimension and entropy}\label{de}
Let $g:Y\to Z$ be a Borel map between two metric spaces $Y$ and $Z$. For a Borel measure $\nu$ on $Y$, write
\[
g\nu=\nu\circ g^{-1},
\]
for the \emph{push-forward measure} of $\nu$ on $Z$ under $g$. For a measure $\nu$ and $x\in\text{supp}(\nu)$, let
\[
D_{\nu}(x)=\lim_{r\to0}\frac{\log\nu(B(x,r))}{\log r},
\]
whenever the limit exists, where $B(x,r)$ is the closed ball of centre $x$ and radius $r$. If for some $\alpha\ge0$ we have $D_{\nu}(x)=\alpha$ for $\nu$-a.e. $x$, we say that $\nu$ is \emph{exact dimensional.}

For $r>0$ and a probability measure $\nu$ supported on a compact subset $A$ of Euclidean space, let
\begin{equation}
\label{eqn:ent}
H_r(\nu)=-\int_{A}\log\nu(B(x,r))\nu(dx)
\end{equation}
be the \emph{r-scaling entropy} of $\nu$. Note that, writing $\mathcal{M}$ for the probability measures supported by $A$, the map $H_r:\mathcal M\to[0,\infty)$ need not be continuous in the weak-$\star$ topology. However, $H_r$ is lower semicontinuous. Indeed, choose continuous functions $f_j:\mathbb{R}^d\to[0,1]$ decreasing pointwise to $\chi_{B(0,r)}$. The continuous functions
\[
\nu\longmapsto\int_A-\log\max\left\{e^{-j},\int_A f_j(x-y)\nu(dy)\right\}\nu(dx)
\]
increase to $H_r(\nu)$. The \emph{lower entropy dimension} of $\nu$ is defined as
\[
\dim_e\nu=\liminf_{r\to0}\frac{H_r(\nu)}{-\log r}.
\]
We use the \emph{lower Hausdorff dimension} of a measure, denoted by
\[
\dim_H\nu=\inf\{\dim_H A:A\text{ is Borel and }\nu(A)>0\}.
\]
Then
\[
\dim_H\nu\le\dim_e\nu,
\]
with equality when $\nu$ is exact dimensional, for details see \cite{FLR}. From \cite[Theorem 2.8]{FengHu} we have that the self-conformal measure $\Phi\mu$ is exact dimensional.
\section{Dimension of projections}
\label{sec:dimofprojs}
We apply the symbolic local-entropy argument to conditional measures on the natural extension. We first construct conformal coordinates in which the conditional cylinder measures satisfy an exact similarity relation. A suspension then places these cylinders at fixed logarithmic scales.

\subsection{Auxiliary estimates}
\label{subsec:auxlem}
We first record the local linearisation estimate.

\begin{lemma}\label{lemma:projdis}
There exists a constant $C_3$ such that for all $\pi\in\Pi_{d,k}$, $n\ge 1$, $i\in\Lambda^n$ and $x,y,z\in V$ one has
\begin{align*}
C_1^{-1}\overline r_i|\pi O_i(z)(x-y)|-C_3\overline r_i&(|z-x|+|z-y|)|x-y|\le |\pi f_i(x)-\pi f_i(y)|\\
\le&\ \overline r_i|\pi O_i(z)(x-y)|+C_3\overline r_i(|z-x|+|z-y|)|x-y|.
\end{align*}
\end{lemma}
\begin{proof}
The chain rule, smoothness on $W$ and uniform contraction give constants $A_1,A_2$ independent of the word $i$ such that
\[
|\log r_i(x)-\log r_i(y)|\le A_1|x-y|,
\qquad \|O_i(x)-O_i(y)\|\le A_2|x-y|
\]
for $x,y\in V$. For the second inequality, expand the difference of the two products of rotation matrices and sum the geometric bounds for the corresponding suffixes. The first follows in the same way from \eqref{eqn:rij}. Consequently,
\[
\|Df_i(x)-Df_i(y)\|\le A_3\overline r_i|x-y|.
\]
Integration along the segment from $y$ to $x$ gives
\[
|f_i(x)-f_i(y)-Df_i(z)(x-y)|
\le A_3\overline r_i(|z-x|+|z-y|)|x-y|.
\]
Apply $\pi$, use $Df_i(z)=r_i(z)O_i(z)$ and $C_1^{-1}\overline r_i\le r_i(z)\le\overline r_i$, and take $C_3=A_3$.
\end{proof}

\subsection{Conformal coordinates indexed by the past}
Fix $x_*\in K$, and put
\[
R(w)=\max_{x\in K}r_w(x),\quad R(\emptyset)=1,\qquad
\rho_0=\max_{i\in\Lambda,x\in K}r_i(x)<1,\quad
r_{\min}=\min_{i\in\Lambda,x\in K}r_i(x)>0.
\]
Bounded distortion implies
\begin{equation}\label{eqn:relative-scales}
C_1^{-1}R(w)R(v)\le R(wv)\le R(w)R(v),\qquad
R(w)\le\overline r_w\le C_1R(w).
\end{equation}
For $p=(\ldots,p_{-1},p_0)\in\Sigma^-$, write $w_n(p)=p_{-n+1}\cdots p_0$ and define
\[
F_{p,n}(x)=R(w_n(p))^{-1}O_{w_n(p)}(x_*)^{-1}
\bigl(f_{w_n(p)}(x)-f_{w_n(p)}(x_*)\bigr),\qquad x\in V.
\]

\begin{lemma}\label{lem:limit-maps}
There are conformal maps $F_p$ such that $F_{p,n}\to F_p$ and $DF_{p,n}\to DF_p$ uniformly on $V$, with convergence uniform in $p$. The limits $F_p$ are uniformly bi-Lipschitz on $V$, satisfy $F_p(x_*)=0$ and $\max_K\|DF_p\|=1$, and depend continuously on $p$ in the $C^1$ topology. For every $p$ and $i\in\Lambda$ there are $b(p,i)\in\mathbb R^d$, $Q(p,i)\in G$ and
\[
r_{\min}/C_1\le a(p,i)\le\rho_0
\]
such that, on a neighbourhood of $K$,
\begin{equation}\label{eqn:limit-covariance}
F_p\circ f_i=b(p,i)+a(p,i)Q(p,i)F_{pi}.
\end{equation}
\end{lemma}
\begin{proof}
It is convenient first to normalise by the derivative at $x_*$:
\[
N_{p,n}(x)=Df_{w_n(p)}(x_*)^{-1}
\bigl(f_{w_n(p)}(x)-f_{w_n(p)}(x_*)\bigr).
\]
Let $g=f_{w_n(p)}$ and let $f_j$ be the map prepended at the next step. The derivative of the next normalised map is
\[
Dg(x_*)^{-1}Df_j(g(x_*))^{-1}Df_j(g(x))Dg(x).
\]
Subtracting $DN_{p,n}(x)=Dg(x_*)^{-1}Dg(x)$, bounded distortion and the uniform H\"older bound for the derivatives of the finitely many maps on $W$ give
\[
\|DN_{p,n+1}-DN_{p,n}\|_{\infty,V}
\le C\bigl(\operatorname{diam}g(V)\bigr)^\theta
\le C\rho_0^{n\theta},\qquad \theta=\min\{1,\epsilon\}>0.
\]
Here the inverse derivatives of the individual maps are bounded on $W$, and $\operatorname{diam}g(V)\le C R(w_n(p))$ by \eqref{eqn:diam} and \eqref{eqn:relative-scales}. Since both maps vanish at $x_*$, convexity of $V$ gives the same bound for their difference in supremum norm. Thus $N_{p,n}$ converges in $C^1(V)$, with a rate uniform in $p$.

Put $c_{p,n}=\max_K\|DN_{p,n}\|$. Then $1\le c_{p,n}\le C_1$ and $F_{p,n}=c_{p,n}^{-1}N_{p,n}$. This proves the asserted convergence and normalisation. The uniform upper and lower Lipschitz bounds follow from \eqref{eqn:diam}. They pass to the limit, so $DF_p$ is nonsingular and $F_p$ is a conformal $C^1$ diffeomorphism onto its image. Continuity in $p$ follows from uniform convergence of the finite-past maps.

At finite depth, direct substitution gives
\[
F_{p,n}\circ f_i=b_{p,n,i}
+\frac{R(w_n(p)i)}{R(w_n(p))}
O_{w_n(p)}(x_*)^{-1}O_{w_n(p)i}(x_*)F_{pi,n+1}.
\]
The translation and the similarity matrix converge: evaluate the identity and its derivative at $x_*$ and use the convergence already proved, noting that $DF_{pi}(x_*)$ is invertible. The scalar ratios lie in $[r_{\min}/C_1,\rho_0]$, by \eqref{eqn:relative-scales}. Passing to the limit proves \eqref{eqn:limit-covariance}. The composition is considered on $V\cap\bigcap_i f_i^{-1}(V)$, an open neighbourhood of $K$; no invariance of $V$ is needed.
\end{proof}

Write $DF_p(x)=s_p(x)B_p(x)$ with $B_p(x)\in G$. If $\omega=(p,ix')$, differentiating \eqref{eqn:limit-covariance} at $\Phi(x')$ gives
\[
B_p(\Phi(ix'))O_i(\Phi(x'))=Q(p,i)B_{pi}(\Phi(x')).
\]
Consequently the change of rotation coordinate
\[
(\omega,O)\longmapsto(\omega,\widetilde O),\qquad
\widetilde O=O B_p(\Phi(\omega^+))^{-1},
\]
conjugates the natural extension of $\sigma_\phi$ to the following map (writing $O$ for the new rotation coordinate):
\begin{equation}\label{eqn:past-skew}
T(\omega,O)=(\sigma\omega,OQ(p,\omega_1)).
\end{equation}
Haar invariance shows that this change of coordinates preserves $\widehat\mu\times\xi$. The natural extension of an ergodic transformation is ergodic, so $T$ is an invertible ergodic transformation under (A2).

\subsection{The zoom-in flow and fixed scales}
Let $p^-$ be the past obtained by deleting $p_0$ and moving the origin one place to the left. Define
\[
c(p)=a(p^-,p_0),\qquad \tau(p)=-\log c(p),\qquad
\Psi_p=c(p)F_p.
\]
Since $c(pi)=a(p,i)$, \eqref{eqn:limit-covariance} becomes
\begin{equation}\label{eqn:past-similarity}
\Psi_p\circ f_i=\widetilde b(p,i)+c(p)Q(p,i)\Psi_{pi}.
\end{equation}
Here $\widetilde b(p,i)=c(p)b(p,i)$. In particular, the contraction $c(p)$ depends only on the current past. Also,
\[
0<\tau_{\min}:=\log(1/\rho_0)\le\tau(p)
\le\log(C_1/r_{\min})=: \tau_{\max}<\infty.
\]

Consider the suspension of $T$ with roof $\tau$:
\[
\mathcal S=\{(\omega,O,t):0\le t<\tau(\omega^-)\},\qquad
d\mathbb P=\chi^{-1}d\widehat\mu(\omega)d\xi(O)dt,
\quad \chi=\int\tau\,d\mu^-.
\]
The flow $S_v$ increases the height at unit speed, with the boundary identification
\[
(\omega,O,\tau(\omega^-))\sim(T(\omega,O),0).
\]
It preserves $\mathbb P$ and is ergodic: a flow-invariant function is constant along the vertical fibres and induces a $T$-invariant function on the base.

Write $z=(p,O,t)$ for the past, rotation and height coordinates, and let $\mathbb P^-$ be their marginal. Since the roof depends only on the past, disintegration gives
\begin{equation}\label{eqn:suspension-disintegration}
d\mathbb P(z,x)=d\mathbb P^-(z)d\mu_p(x),\qquad
d\mathbb P^-(p,O,t)=\chi^{-1}d\mu^-(p)d\xi(O)
\mathbf1_{[0,\tau(p))}(t)dt.
\end{equation}
Define
\[
A_z=e^tO\Psi_p,\qquad \Phi_z=A_z\circ\Phi,\qquad M(z)=\Phi_z\mu_p.
\]
The Gibbs measure $\mu$ has full support on $\Sigma^+$, so \eqref{eqn:past-equivalence} implies $\operatorname{supp}M(z)=A_z(K)$ for almost every $z$. We also regard $M$ as a function on $\mathcal S$ by forgetting the future. These measures are supported in a fixed closed ball $B$, since $e^tc(p)\le1$ and $F_p(K)$ is uniformly bounded. By \eqref{eqn:past-equivalence} and Lemma \ref{lem:limit-maps}, for almost every $z$ they are exact dimensional of dimension
\[
\alpha=\dim_H\Phi\mu.
\]

We choose a sampling interval $\ell>\max\{\tau_{\max},\log4\}$ such that $S_{q\ell}$ is ergodic for every integer $q\ge1$, and put $\rho=e^{-\ell}$. Such a choice introduces no additional hypothesis. Indeed, an ergodic measurable flow on a standard probability space has at most countably many nonzero eigenfrequencies, since the corresponding eigenfunctions are orthogonal in the separable space $L^2$. If $S_v$ is not ergodic for some $v>0$, choose a nonzero, mean-zero $S_v$-invariant function $f\in L^2$. The $L^2$-valued function $t\mapsto f\circ S_t$ is $v$-periodic. Ergodicity of the flow makes its constant Fourier coefficient zero, so some nonconstant coefficient is nonzero and gives a flow eigenfunction of frequency $2\pi j/v$, with $j\ne0$. Thus the exceptional nonzero times are countable. Avoiding these times divided by positive integers gives the required $\ell$.

Fix $q\ge1$ and $z=(p,O,t)$. For a future $x$, let $p^{(j)}=px_1\cdots x_j$ for $j\ge1$ and $p^{(0)}=p$. For $v\ge0$, put
\[
A_m(p,x)=\sum_{j=0}^{m-1}\tau(p^{(j)}),\qquad A_0=0,
\qquad N_v(z,x)=\max\{m:A_m(p,x)\le t+v\}.
\]
At time $v=nq\ell$, the stopping condition is
\begin{equation}\label{eqn:redefsymsp}
A_m(p,x)\le t+nq\ell<A_{m+1}(p,x).
\end{equation}
It depends only on the first $m$ future symbols, because $\tau(p^{(m)})$ is past measurable. The corresponding cylinders form a finite partition $\mathcal P_{n,q}(z)$ of $\Sigma^+$. These partitions are nested, start with the trivial partition, and generate the Borel sets. Since $q\ell>\tau_{\max}$, every successive stopping word strictly extends its predecessor. The number of new symbols at each step is bounded in terms of $q\ell,\tau_{\min},\tau_{\max}$, so the resulting tree has uniformly bounded branching. Endow its boundary, identified with $\Sigma^+$, with the $\rho^q$-tree metric, and denote it by $X_{q,z}$.

For a cylinder $[u]\in\mathcal P_{n,q}(z)$, write $z_u$ for the past, rotation and height coordinates at time $nq\ell$ along any future beginning with $u$. These coordinates depend only on $z,u$. Iterating \eqref{eqn:past-similarity} gives a vector $b_{z,u}$ such that
\begin{equation}\label{eqn:exact-zoom}
\Phi_z(uv)=b_{z,u}+\rho^{qn}\Phi_{z_u}(v)
\qquad(v\in\Sigma^+).
\end{equation}
Indeed, the product of the contractions is $e^{-A_{|u|}}$ and the final height is $t+nq\ell-A_{|u|}$. Thus the factor left after comparing the two height normalisations is exactly $e^{-nq\ell}$.

\begin{lemma}\label{lemma:zoom}
There is $L\ge1$, independent of $z,q,n$, such that, writing $\Phi_z[u]=\Phi_z([u])$,
\begin{equation}\label{eqn:cylinderdiam}
\operatorname{diam}\Phi_z[u]\le L\rho^{qn}
\qquad([u]\in\mathcal P_{n,q}(z)).
\end{equation}
For $\mathbb P^-$-almost every $z$, every such $u$ and every $\pi\in\Pi_{d,k}$,
\begin{equation}\label{eqn:zoom-entropy}
H_{\rho^{q(n+1)}}\bigl(\pi\Phi_z(\mu_p)_{[u]}\bigr)
=H_{\rho^q}\bigl(\pi M(z_u)\bigr).
\end{equation}
\end{lemma}
\begin{proof}
The diameter assertion follows from \eqref{eqn:exact-zoom} and the common bound for $\Phi_{z_u}(\Sigma^+)$. By \eqref{eqn:past-covariance}, the push-forward of $(\mu_p)_{[u]}$ under deletion of $u$ is exactly $\mu_{pu}$. Hence \eqref{eqn:exact-zoom} also holds for the conditional measures. Translation invariance and the scaling rule for entropy give \eqref{eqn:zoom-entropy}.
\end{proof}

\subsection{Ergodic averages}
Let $\mathcal M$ be the compact space of probability measures on $B$ with the weak-$\star$ topology. For $\pi\in\Pi_{d,k}$ define
\begin{equation}\label{eqpi}
e_q(\pi,\nu)=\frac{H_{\rho^q}(\pi\nu)}{q\log(1/\rho)},\qquad
E_q(\pi)=\int e_q(\pi,M(z))\,\mathbb P^-(dz).
\end{equation}
These quantities are nonnegative and bounded by $k+O(1/q)$, uniformly in $\pi$ and $\nu\in\mathcal M$. This follows by covering the projection of $B$ by $O(\rho^{-qk})$ sets of diameter at most $\rho^q$. Haar invariance in the $O$ coordinate implies that $E_q(\pi)$ is independent of $\pi$.

\begin{proposition}\label{prop:ergcge}
For $\mathbb P$-almost every $\zeta\in\mathcal S$, simultaneously for every integer $q\ge1$ and $h\in C(\mathcal M)$,
\[
\lim_{N\to\infty}\frac1N\sum_{n=0}^{N-1}h(M(S_{nq\ell}\zeta))
=\int h(M(z))\,\mathbb P^-(dz).
\]
On this same full-measure set, simultaneously for every integer $q\ge1$ and $\pi\in\Pi_{d,k}$,
\begin{equation}\label{eqn:entropy-averages}
\liminf_{N\to\infty}\frac1N\sum_{n=1}^N
H_{\rho^q}(\pi M(S_{nq\ell}\zeta))
\ge q\log(1/\rho)E_q(\pi).
\end{equation}
\end{proposition}
\begin{proof}
Choose a countable dense family in $C(\mathcal M)$. Apply the ergodic theorem to its members composed with $M$, for each of the countably many ergodic transformations $S_{q\ell}$, and intersect the resulting full-measure sets. Uniform approximation gives convergence for all $h\in C(\mathcal M)$, or equivalently weak convergence of the empirical distributions of $M(S_{nq\ell}\zeta)$. For every $\pi$, the function $\nu\mapsto H_{\rho^q}(\pi\nu)$ is nonnegative and lower semicontinuous. The Portmanteau theorem therefore gives \eqref{eqn:entropy-averages} for all projections on the same set. Boundedness of the entropy at a fixed scale permits shifting the averages by one term.
\end{proof}

\subsection{Lifting at fixed contraction scales}
\begin{remark}\label{rmk:faithful-decay}
The entropy comparison in \cite[Proposition 5.3]{HochShmer} requires the corrected decay condition in Definition 5.1: an $n$-cylinder image must contain a ball of radius $C^{-1}s^n$ and be contained in a ball of radius $Cs^n$, with $C$ independent of $n$. The printed radii $(C^{-1}s)^n$ and $(Cs)^n$ do not give an entropy error uniform in $n$. The cube construction in their Theorem 5.4 satisfies the corrected condition. We give the needed estimate below, keeping the same tree parameter $s=\rho^q$ throughout.
\end{remark}

An $s$-tree is the boundary of a rooted tree with uniformly bounded branching, endowed with the metric determined by common prefix length and the parameter $s$. A tree morphism preserves levels and prefixes. For a probability measure $\eta$ and the partition $\mathcal Q_n$ into $n$-cylinders, write
\[
H(\eta,\mathcal Q_n)=-\sum_{[a]\in\mathcal Q_n}\eta([a])\log\eta([a]),
\]
with $0\log0=0$.

\begin{lemma}\label{lem:adaptive-lifting}
Let $0<s<1/4$, $L\ge1$, and let $f:X\to\mathbb R^k$ be an $L$-Lipschitz map on an $s$-tree. There are an $s$-tree $Y$, a tree morphism $g:X\to Y$ and a continuous map $F:Y\to\mathbb R^k$ such that $f=F\circ g$. There are constants $B_0,B_1$, depending only on $k,L$, with the following properties:
\begin{enumerate}[(i)]
\item for every $n$-cylinder $[a]\subset X$ and every probability measure $\nu$ supported on $[a]$,
\begin{equation}\label{eqn:cq}
\left|H(g\nu,\mathcal Q_{n+1})-H_{s^{n+1}}(f\nu)\right|\le B_0;
\end{equation}
\item for every probability measure $\eta$ on $Y$,
\begin{equation}\label{eqn:lift-dimension}
\dim_H F\eta\ge\dim_H\eta-\frac{B_1}{\log(1/s)}.
\end{equation}
\end{enumerate}
In particular, the constants are independent of $s$ and the cylinder depth.
\end{lemma}
\begin{proof}
This is the cube construction of \cite[Theorem 5.4]{HochShmer}, with its scale bounds made explicit. Put $j_n=\lfloor n\log_2(1/s)\rfloor$ and $\ell_n=2L2^{-j_n}$, so
\[
2Ls^n\le\ell_n<4Ls^n.
\]
Start with a closed cube of side $\ell_0$ containing $f(X)$. Cover each level-$n$ cube by cubes of side $\ell_{n+1}$ contained in it, using a grid of spacing $\ell_{n+1}/2$. The ratio $\ell_{n+1}/\ell_n$ is dyadic and at most $2s<1$. The cover has multiplicity at most $M=3^k$, and every subset of the parent cube of diameter at most $\ell_{n+1}/2$ is contained in a child cube.

The tree of nested cubes defines $Y$ and its coding map $F$. Each cylinder image is its entire cube. Assign source cylinders to cubes inductively: since $\operatorname{diam}f[a]\le Ls^n\le\ell_n/2$, a cube can be chosen at every step, consistently with the parent assignment. This gives $g$ and $f=F\circ g$.

To prove (i), put $r=s^{n+1}$. The child cubes of the assigned parent have sides between $2Lr$ and $4Lr$. The grid implies that each has at most $N=N(k,L)$ neighbours at distance at most $r$, independently of $s,n$. Let $p_i$ be the masses assigned by $g\nu$ to these children and $\mathcal N(i)$ their neighbour sets. For a point assigned to child $i$,
\[
(f\nu)(B(y,r))\le\sum_{j\in\mathcal N(i)}p_j.
\]
By Jensen's inequality and symmetry of the neighbour relation,
\begin{align*}
H(g\nu,\mathcal Q_{n+1})-H_r(f\nu)
&\le\sum_{i:p_i>0}p_i\log\frac{\sum_{j\in\mathcal N(i)}p_j}{p_i}\\
&\le\log\sum_{i:p_i>0}\sum_{j\in\mathcal N(i)}p_j\le\log N.
\end{align*}
Conversely, the child diameters are at most $4\sqrt{k}Lr$, so
\[
H_{4\sqrt{k}Lr}(f\nu)\le H(g\nu,\mathcal Q_{n+1}).
\]
The fixed-factor scale comparison \cite[Lemma 3.5]{HochShmer} proves (i).

At level $n$, the cube multiplicity is at most $M^n$. Every cube meeting a ball of radius $s^n$ has volume at least $(2Ls^n)^k$ and is contained in a ball of radius $(1+4\sqrt{k}L)s^n$. A volume comparison bounds the number of such cubes by $A M^n$, where $A$ depends only on $k,L$.
If $0<\theta<\dim_H\eta$, the local-dimension characterisation gives sets $E_j$ covering almost all of $\eta$ such that
\[
(\eta|_{E_j})([a])\le j s^{n\theta}
\]
for every level-$n$ cylinder. Hence
\[
F(\eta|_{E_j})(B(y,s^n))\le AjM^n s^{n\theta}.
\]
It follows that $\dim_H F\eta\ge\theta-\log M/\log(1/s)$. Letting $\theta\uparrow\dim_H\eta$ proves (ii), with $B_1=\log M$. The case $\dim_H\eta=0$ is immediate.
\end{proof}

\begin{proposition}\label{prop:local-entropy-dimension}
Let $f:X\to\mathbb R^k$ be $L$-Lipschitz on an $s$-tree, where $L\ge1$ and $0<s<1/4$. Let $\lambda$ be a probability measure on $X$. If, for $\lambda$-almost every $x$,
\[
\liminf_{N\to\infty}\frac1N\sum_{n=1}^N
H_{s^{n+1}}(f\lambda_{[x|_n]})\ge A,
\]
then there is $B=B(k,L)$, independent of $s$, such that
\[
\dim_H f\lambda\ge\frac{A-B}{\log(1/s)}.
\]
\end{proposition}
\begin{proof}
By \eqref{eqn:cq}, the averaged Shannon entropy of the next target child, conditional on the source prefix, has lower limit at least $A-B_0$. The local entropy averages theorem for tree morphisms, \cite[Theorem 4.4]{HochShmer}, gives
\[
\dim_H g\lambda\ge\frac{A-B_0}{\log(1/s)}.
\]
Now apply \eqref{eqn:lift-dimension} and take $B=B_0+B_1$. Using child-partition entropy avoids any convention about open or closed balls on the tree.
\end{proof}

\subsection{The projection lower bound}
\label{subsec:lowerbound}
\begin{theorem}\label{thm:lowerbound}
There is $B$, depending only on the IFS and $k$, such that for $\mathbb P^-$-almost every $z$, simultaneously for all integers $q\ge1$ and all $\pi\in\Pi_{d,k}$,
\begin{equation}\label{eqn:dimlowbound}
\dim_H\pi M(z)\ge E_q(\pi)-\frac{B}{q\log(1/\rho)}.
\end{equation}
\end{theorem}
\begin{proof}
Disintegrate the common full-measure set of Proposition \ref{prop:ergcge} using \eqref{eqn:suspension-disintegration}. Thus, for almost every $z=(p,O,t)$, its conclusions hold for $\mu_p$-almost every future, simultaneously for all $q,\pi$. Lemma \ref{lemma:zoom} gives
\begin{equation}\label{eqn:entropy-input}
\liminf_{N\to\infty}\frac1N\sum_{n=1}^N
H_{\rho^{q(n+1)}}\bigl(\pi\Phi_z(\mu_p)_{[u_n(x)]}\bigr)
\ge q\log(1/\rho)E_q(\pi),
\end{equation}
where $[u_n(x)]$ is the atom of $\mathcal P_{n,q}(z)$ containing $x$. By \eqref{eqn:cylinderdiam}, $\pi\Phi_z$ is $L$-Lipschitz on $X_{q,z}$ with $L$ independent of $q,z,\pi$. Apply Proposition \ref{prop:local-entropy-dimension} with $s=\rho^q$, $\lambda=\mu_p$ and $A=q\log(1/\rho)E_q(\pi)$.
\end{proof}

\subsection{Projection theorems}
\begin{theorem}\label{thm:dimequalse}
The limit $E(\pi)=\lim_{q\to\infty}E_q(\pi)$ exists for every $\pi\in\Pi_{d,k}$. Moreover:
\begin{enumerate}[(i)]
\item for each fixed $\pi$, for $\mathbb P^-$-almost every $z$,
\[
\dim_e\pi M(z)=\dim_H\pi M(z)=E(\pi);
\]
\item for $\mathbb P^-$-almost every $z$, simultaneously for every $\pi\in\Pi_{d,k}$,
\[
\dim_H\pi M(z)\ge E(\pi).
\]
\end{enumerate}
\end{theorem}
\begin{proof}
Theorem \ref{thm:lowerbound} gives, on a common full-measure set,
\begin{equation}\label{eqn:qinf}
\dim_H\pi M(z)\ge\limsup_{q\to\infty}E_q(\pi)
\end{equation}
for every $\pi$. For fixed $\pi$, use $\dim_H\nu\le\dim_e\nu$ and Fatou's lemma in \eqref{eqpi} to obtain
\begin{align*}
\limsup_{q\to\infty}E_q(\pi)
&\le\int\dim_H\pi M(z)\,\mathbb P^-(dz)\\
&\le\int\dim_e\pi M(z)\,\mathbb P^-(dz)
\le\liminf_{q\to\infty}E_q(\pi).
\end{align*}
The lower entropy dimension can be computed along $\rho^q$, by monotonicity of $H_r$ and the fixed ratio of successive scales. All inequalities are therefore equalities. Together with \eqref{eqn:qinf}, this proves both assertions.
\end{proof}

Put $\beta=\min\{k,\alpha\}=\min\{k,\dim_H\Phi\mu\}$.
\begin{corollary}\label{cor:dim}
For every $\pi$, $E(\pi)=\beta$. For $\mathbb P^-$-almost every $z$, simultaneously for every $\pi\in\Pi_{d,k}$,
\[
\dim_H\pi M(z)=\beta.
\]
\end{corollary}
\begin{proof}
For almost every $p$ and each admissible $t$, the measure $e^t\Psi_p\Phi\mu_p$ is exact dimensional of dimension $\alpha$. The projection theorem for measures \cite{HuntKaloshin}, applied to its Haar rotations, gives $\dim_H\pi M(p,O,t)=\beta$ for each fixed $\pi$ and almost every $O$. Fubini's theorem and Theorem \ref{thm:dimequalse}(i) imply $E(\pi)=\beta$. Part (ii) gives the lower bound for all projections on one full-measure set. The opposite inequality follows from the Lipschitz property of projections and exact dimensionality of $M(z)$.
\end{proof}

\begin{proposition}\label{prop:nonli}
For $\mathbb P^-$-almost every $z$, for every $C^1$ map $h$ from a neighbourhood of $\operatorname{supp}M(z)$ to $\mathbb R^k$ whose derivative has rank $k$ on this support,
\[
\dim_H hM(z)=\beta.
\]
Consequently, every $C^1$ map $h$ from a neighbourhood of $K$ to $\mathbb R^k$ without singular points on $K$ satisfies
\[
\dim_H h\Phi\mu=\min\{k,\dim_H\Phi\mu\}.
\]
\end{proposition}
\begin{proof}
Choose $z=(p,O,t)$ in the full-measure set used in Theorem \ref{thm:lowerbound}, so that \eqref{eqn:entropy-input} holds for $\mu_p$-almost every future, for all $q$ and $\pi$. Fix $h$ as in the first assertion. The least nonzero singular value of $Dh(y)$ has a positive lower bound $m_h$ on the compact support. Let $P_y$ be the orthogonal projection onto $(\ker Dh(y))^\perp$. Then
\[
|Dh(y)(v-w)|\ge m_h|P_y(v-w)|.
\]
Fix $q$ and put $s=\rho^q$. If $y=\Phi_z(x)$ and $u_n=u_n(x)$, uniform continuity of $Dh$ and \eqref{eqn:cylinderdiam} give
\[
|h(v)-h(w)-Dh(y)(v-w)|\le\delta_n s^n
\qquad(v,w\in\Phi_z[u_n]),
\]
where $\delta_n\to0$, uniformly in $x$ for fixed $q$. Segments between such points lie in the domain of $h$ for all sufficiently large $n$. Taking $n$ large enough that $\delta_n\le s$, we have
\[
|h(v)-h(w)|\le s^{n+1}
\quad\Longrightarrow\quad
|P_y(v-w)|\le 2s^{n+1}/m_h.
\]
Integrating this ball-mass comparison and using \cite[Lemma 3.5]{HochShmer} gives
\[
H_{s^{n+1}}\bigl(h\Phi_z(\mu_p)_{[u_n]}\bigr)
\ge H_{s^{n+1}}\bigl(P_y\Phi_z(\mu_p)_{[u_n]}\bigr)-C_h,
\]
with $C_h$ independent of $q,n,x$. Along a fixed future $x$, the projection $P_y$ is fixed. Since \eqref{eqn:entropy-input} holds simultaneously for every projection, and $E_q$ is independent of the projection, the lower limit of the averages on the left is at least $q\log(1/\rho)E_q-C_h$.

The map $h$ is Lipschitz on the compact support, so $h\Phi_z$ is $L_h$-Lipschitz on $X_{q,z}$ with $L_h$ independent of $q$. Proposition \ref{prop:local-entropy-dimension} yields
\[
\dim_H hM(z)\ge E_q-\frac{C_h+B(k,L_h)}{q\log(1/\rho)}.
\]
Letting $q\to\infty$ and using Corollary \ref{cor:dim} gives the lower bound $\beta$. The upper bound follows from the Lipschitz property. The same typical set of $z$ works for every $h$, since its definition used only the simultaneous projection averages. The argument also applies to a Lipschitz map which is $C^1$ off an $M(z)$-null set, provided that its derivative has rank $k$ there and the least nonzero singular value has a positive uniform lower bound. In that case the Taylor estimate is used pointwise along each typical future, with the initial depth allowed to depend on its image point.

Finally, fix any such $z$. The map $A_z=e^tO\Psi_p$ is a conformal $C^1$ diffeomorphism on $V$, and $M(z)=A_z\Phi\mu_p$. For a given nonsingular $h$ on a neighbourhood of $K$, apply the first assertion to $h\circ A_z^{-1}$ on a neighbourhood of $A_z(K)$. It follows that $\dim_H h\Phi\mu_p=\beta$. By \eqref{eqn:past-equivalence}, $h\Phi\mu_p$ and $h\Phi\mu$ are equivalent, so their lower Hausdorff dimensions agree. This proves the second assertion, and hence Theorem \ref{thm:main} and Corollary \ref{thm:nonli}.
\end{proof}

\section*{Acknowledgements}
We would like to thank Julien Barral, Kenneth Falconer, Jonathan Fraser and Thomas Jordan for helpful discussions and comments. We are grateful to Ariel Rapaport for pointing out the error in the decay condition in \cite[Definition 5.1]{HochShmer} and its implications for the proof in the published version of this paper. Part of this work was completed while the second author was visiting the Institut Mittag-Leffler as part of the `Fractal Geometry and Dynamics' Programme. He would like to thank the organisers and staff members for their hospitality.

\end{document}